\documentclass[12pt,letterpaper]{article}
\usepackage[square,sort,numbers]{natbib}
\usepackage{bbm}
\usepackage{bm}
\usepackage{chngcntr}
\usepackage{multirow}

\usepackage[colorlinks=true,urlcolor=blue,citecolor=blue,linkcolor=blue,bookmarks=true,bookmarksopen=false]{hyperref}

\usepackage[ruled,linesnumbered,vlined]{algorithm2e}
\usepackage{algorithmic}
\usepackage{booktabs}
\usepackage{hhline}

\newcommand{\vc}{\mathbf{c}}
\newcommand{\va}{\mathbf{a}}
\newcommand{\vX}{\mathbf{X}}
\newcommand{\vY}{\mathbf{Y}}

\newcommand{\vq}{\mathbf{q}}

\newcommand{\vw}{\mathbf{w}}

\newcommand{\vv}{\mathbf{v}} %vector v
\newcommand{\vz}{\mathbf{z}} %vector z
\newcommand{\vx}{\mathbf{x}} %vector x
\newcommand{\vy}{\mathbf{y}} %vector y

\newcommand{\veta}{\bm{\eta}} %vector y
\newcommand{\vb}{\mathbf{b}} %vector v

\newcommand{\vs}{\mathbf{s}} %vector v

\newcommand{\vk}{\mathbf{k}}

\newcommand{\vvartheta}{\bm{\vartheta}}

\newcommand{\Min}{\operatorname{Min}}

\renewcommand{\Omega}{\varOmega}
\newcommand{\R}{\mathbbm{R}}
\newcommand{\p}{\mathbbm{P}}
\newcommand{\E}{\mathbbm{E}}

\newcommand{\X}{\mathcal{X}}

\newcommand{\cvar}{\operatorname{CVaR}}
\newcommand{\wcvar}{\operatorname{WCVaR}}
\newcommand{\var}{\operatorname{VaR}}

\newcommand{\mvar}{\operatorname{MVaR}}
\newcommand{\mcvar}{\operatorname{MCVaR}}
\newcommand{\vmcvar}{\operatorname{VMCVaR}}

\newcommand{\vG}{\mathbf{G}}

\newcounter{commentcounter}
\long\def\symbolfootnote[#1]#2{\begingroup%
\def\thefootnote{\fnsymbol{footnote}}\footnote[#1]{#2}\endgroup}

\usepackage{sibarticle}
\usepackage{threeparttable}
\usepackage{pdflscape}

\allowdisplaybreaks
\newcommand{\ignore}[1]{}

\title{Vector-Valued Multivariate Conditional Value-at-Risk}
\ShortTitle{Vector-Valued Multivariate $\cvar$} \ShortAuthors{ Merakl{\i}, K\"{u}\c{c}\"{u}kyavuz} \NumberOfAuthors{1}
\FirstAuthor{Merve Merakl{\i}, Simge K\"{u}\c{c}\"{u}kyavuz}
\FirstAuthorAddress{Industrial and Systems Engineering, University
of Washington, Seattle, WA, U.S.A.\\ {\tt merakli@uw.edu, simge@uw.edu}}


\begin{document}

\maketitle \centerline{\today}

\begin{abstract}
In this study, we propose a new definition of multivariate conditional value-at-risk (MCVaR) as a set of vectors for discrete probability spaces. We explore the properties of the vector-valued MCVaR  (VMCVaR) and show  the advantages of VMCVaR over   the existing definitions given for continuous random variables  when adapted to the discrete case. 
\end{abstract}

\section{Introduction}

\textit{Conditional value-at-risk} (CVaR) is a widely used tool, especially in financial optimization, for assessing the risk associated with a certain decision. The value of CVaR at an uncertainty level $p$ is the expected outcome given that it is unfavorable compared to at least 100$p$\% of the possible realizations. The threshold value corresponds to the $p$-quantile, which is also known as the \textit{value-at-risk} (VaR). VaR is a suitable risk measure for the cases where the aim is to avoid unfavorable outcomes  with high probability. 
However, it does not measure the magnitude of the unfavorable outcomes. 
To address this,  CVaR has been introduced to quantify the expected value of the undesired outcomes as well as computing the VaR as a by-product  \cite[]{rockafellar2000optimization}. 

CVaR is first introduced as a risk measure for univariate random variables. Due to its  desirable properties such as coherence and law invariance, CVaR is widely incorporated into optimization problems to minimize or limit the risk of the corresponding decisions (see e.g., \cite{rockafellar2000optimization}, \cite{rockafellar2002conditional}, \cite{fabian2007algorithms}). However, in many real life problems, decision makers are interested in measuring the risk arising from multiple factors rather than a single outcome. One way of extending the univariate definitions of VaR and CVaR to the multivariate case is to use a scalarization vector for turning the random outcome vector into a scalar. However,  the relative importance of criteria is usually ambiguous.  The weights of criteria with respect to each other are not always quantifiable as a unique vector,  as they may be subject to conflicting opinions of  a group of decision makers. To address this, recent work considers a robust approach, where the set of possible scalarization vectors is assumed to be known and often assumed to be polyhedral or convex, and a worst-case scalarization vector in this set is used to scalarize the multivariate random vector  \cite{noyan2013optimization,liu2015robust,kuccukyavuz2016cut,noyan2016optimization,NMK17}.  These papers not only define the concept of polyhedral multivariate CVaR for finite discrete distributions, but also give optimization models and solution methods when the multivariate random outcome vector is a function of the decisions. For computability of these risk measures, the solution methods rely on sampling from the true distribution. In this context, it is natural to consider finite discrete distributions, described with a finite number of scenarios. 
In this paper, instead of a scalarized real-valued representation of the multivariate CVaR,  we introduce a vector-valued  risk measure without the need for the specification of an ambiguity set describing the possible weights.

The first challenge in defining a  multivariate CVaR is the determination of the $p$-quantile (multivariate VaR). While the $p$-quantile corresponds to a single real value in the univariate setting, $p$-quantile of a random vector may point to several vectors in the multivariate context. There are several studies addressing this issue. \citet{prekopa1990dual} proposes a multivariate definition of VaR as a set of vectors, called $p$-Level Efficient Points ($p$LEPs), rather than a single value for arbitrary multivariate distributions. \citet{cousin2013multivariate} propose an alternative approach for the continuous distribution case that computes a single vector as the expectation of the boundary surface of the multivariate quantile function. \citet{torres2015directional}, on the other hand, introduce a direction parameter and compute the multivariate VaR vector in that direction. \citet{di2015multivariate} and \citet{adrian2016covar} also provide single, vector-valued definitions of multivariate VaR by conditioning on the information on  certain criteria. As in the univariate case, these multivariate VaR definitions are only concerned with the probability of having favorable outcomes, and not the magnitude of the unfavorable outcomes. 

The second challenge in defining a  multivariate CVaR is related with the ambiguity in the characterization of \textit{undesirable} outcomes in the multivariate case. \citet{lee2013properties} provide a single and real-valued definition of multivariate CVaR that classifies a vector as undesirable if it is not better than or equivalent to any $p$LEP as defined in \cite{prekopa1990dual}. \citet{cousin2014multivariate} propose a single, vector-valued CVaR definition such that realizations worse than at least one $p$LEP with respect to every criteria are assumed to be undesirable. The details of these studies providing a definition for multivariate $\cvar$ will be discussed in Section \ref{sec:comp}.

This study is dedicated to defining a vector-valued multivariate CVaR that specifically targets discrete random variables without characterizing a set that describes the relative importance of criteria in advance. Throughout the paper, smaller values of random variables as well as smaller values of risk measures are assumed to be preferable.  We propose a new definition for multivariate CVaR and explore its properties in Section \ref{sec:def}. In Section \ref{sec:comp}, we review the existing definitions, and demonstrate their shortcomings in the case of finite discrete distributions. We show that our new definition overcomes these problems,  which leads to  a unified context for multivariate CVaR. 

\section{Multivariate CVaR} \label{sec:def}

We use the multivariate VaR definition in \cite{prekopa1990dual} as the $p$-quantile function for our multivariate CVaR definition. \citet{prekopa1990dual} defines multivariate VaR using $p$LEPs for the discrete distribution case.

\begin{definition}[\citet{prekopa1990dual}] \label{def:plep}
Let $\vX \in \R^d$ be a random vector and $F$ its c.d.f. The vector $\vs \in \R^d$ is a $p$LEP (or a $p$-efficient point) of the distribution of $\vX$, if $F(\vs)\geq p$ and there is no $\vy \leq \vs,\ \vy \neq \vs$ such that $F(\vy) \geq p$.
\end{definition}

In the multivariate context, there may be several vectors ($p$LEPs) satisfying the above conditions. Hence $\mvar_p(\vX)$ corresponds to the set of such vectors. \citet{dentcheva2000concavity} show that the set $\mvar_p(\vX)$ is finite in the case that $\vX$ has a finite discrete distribution. The problem of finding elements of $\mvar_p(\vX)$ can be seen as the following multi-objective optimization problem:
\begin{subequations}\label{form:optMVar}
\begin{align}
\var_p(X) = &\min\quad \vv\\
 &\text{s.t.} \quad \p(\vX \leq \vv) \geq p.
\end{align}
 \end{subequations}
The elements in $\mvar_p(\vX)$ correspond to the Pareto efficient points of problem (\ref{form:optMVar}) with joint chance constraints. There is no polynomial-time algorithm to solve \eqref{form:optMVar} unless P=NP. \citet{prekopa2012multivariate} reviews several algorithms to enumerate the exponential set $\mvar_p(\vX)$. As mentioned earlier, $\mvar_p(\vX)$ concerns the $p$-quantile of the multivariate risk only. To measure the magnitude of the risk associated with the $(1-p)100$\% worst outcomes of multiple risk factors, a multivariate analogue of CVaR is needed. 

Next, we propose a new definition for multivariate CVaR as a set of vectors, each associated with a $p$LEP in $\mvar$. In our definition, outcomes exceeding a $p$LEP in at least one criterion are considered as undesirable with respect to that $p$LEP.  

\begin{definition}[Vector-Valued Multivariate Conditional Value-at-Risk] \label{def:1}
Assume that $\vX \in \R^d$ is a random vector with a set of $p$LEPs $\mvar_p(\vX)$ at confidence level $p$. The vector-valued multivariate CVaR of $\vX$ at level $p$, denoted as $\vmcvar_p(\vX)$ is defined as,
\begin{align}
\vmcvar_p(\vX) = & \Min\{\mcvar_p(\vX,\veta):\ \veta \in \mvar_p(\vX)\},\label{def:mcvar1}
\end{align}
where 
$$\mcvar_p(\vX,\veta) = \veta+\frac{1}{1-p}\E[(\vX-\veta)_+]$$
and $[(\vX-\veta)_+]_i = \max(0,X_i-\eta_i)$ for all $i \in [d]:=\{1,\dots,d\}$.
\end{definition} 

The $\Min$ operator in Definition \ref{def:1} ensures that  an element of $\vmcvar_p(\vX) $ is non-dominated by another element in this set, i.e., the elements of $\vmcvar_p(\vX) $  form an efficient frontier. (In the context that smaller values are preferable,  a vector $\vx\in \R^d$ is {\it non-dominated} by a vector $\vy\in \R^d$ if $y_i>x_i$ for some $i\in[d]$. In addition, $\vx$ dominates $\vy$ if $y_i\ge x_i$ for all $i\in[d]$.) Note that if $X$ is a  univariate random variable with a finite discrete distribution, then there is a unique element $\mcvar_p(X,\eta)$ for  $\eta = \var_p(X)$, hence the Min operator is not needed.    To see why this operator is needed in the multivariate case, we provide an example next. 

\begin{example}
 Let $\vX$  be  a bivariate random vector with support
$\vX\in  \{(4,1.5),(1,3),(2,5),(2,3),(3,1)\},$
such that each realization is equally likely. At confidence level $p=0.6$, $\mvar_p(\vX)=\{(3,3),(2,5)\}$.  Using this information, we compute $\mcvar_p(\vX, (3,3))=(3.5,4)$, and  $\mcvar_p(\vX, (2,5))=(3.5,5)$.  Clearly, $\mcvar_p(\vX, (3,3))\le \mcvar_p(\vX, (2,5))$, hence we let $\vmcvar(\vX)=\{( 3.5,4)\}$. 
\end{example}

The definition of VMCVaR is  analogous to the univariate definition of CVaR  for a univariate random variable $V$ given by 
$$
\cvar_{p}(V) = \min\{\eta+\frac{1}{1-p} \E(V-\eta)_+ \}.$$ 
Suppose that $V$ follows a finite discrete distribution. In other words, there is a finite number of scenarios, $n$, where $v_i$ represents the realization of $V$ under scenario $i\in[n]$ with probability $q_i$. Then an equivalent representation of univariate CVaR is given by 
a linear program (LP)
\begin{align}
\label{cvardef}\cvar_{p}(V) &= \min \{\eta+\frac{1}{1-p}\sum_{i\in [n]}q_i
w_i~:~w_i\geq v_i-\eta,~\forall~i\in [n],\quad \vw\in
\R_+^{n},~\eta\in\R\}.
\end{align}
It is well known that at an optimal solution to this LP, $\eta=\var_p(V)$. 
Now we consider an analogous representation of $\vmcvar_p(\vX)$. Let $\vx^s=(x_1^s,\dots,x_d^s)$ be the realization of $\vX$ under scenario $s$ with probability $q_s$ for $s\in [n]$. Then,
\begin{subequations}\label{eq:cvarcont}
\begin{align}
\vmcvar_p(\vX)=\min\quad & \veta+\frac{1}{1-p}\E[(\vX-\veta)_+] \\
\text{s.t.}\quad & \p(\vX \leq \veta) \geq p, \label{c1}
\end{align}
\end{subequations}
or equivalently
\begin{subequations}\label{eq:cvardisc}
\begin{align}
\min\quad & \veta+\frac{1}{1-p}\sum_{s \in [n]} q_s \vw^s  \\
\text{s.t.}\quad & w^s_i \geq \vx^s_i - \eta_i, \quad i \in [d],\ s \in [n],  \\
& \vw^s \geq \mathbf{0},\quad  s \in [n],  \\
& \sum_{s \in [n]} q^s \beta_s \leq 1-p,  \\
& \vx_i^s \leq \eta_i + M_{is} \beta_s, \quad  i \in [d],\ s \in [n], \label{eq:bigM} \\
& \beta_s \in \{0,1\}, \quad s \in [n].
\end{align}
\end{subequations}
Here $\veta$ is in the set $\mvar_p(\vX)$, $\beta_s=0$ if under scenario $s$ we have $\vx^s\le \veta$
and
$M_{is},  i \in [d],\ s \in [n]$ is a large enough number so that constraint \eqref{eq:bigM} is trivially satisfied if $\beta_s=1$. The complexity of the multivariate CVaR definition is evident from comparing the formulations for the univariate and multivariate case; the former is an LP \eqref{cvardef}, whereas the latter is a multiobjective mixed-integer program \eqref{eq:cvardisc} or a chance-constrained program \eqref{eq:cvarcont}. However, the main difficulty lies at identifying the MVaR values. Once these are identified, say using the methods reviewed in \cite{prekopa2012multivariate}, calculating $\mcvar_p(\vX,\eta)$ for each $\eta\in\mvar_p(\vX)$ is polynomial in the number of scenarios, $n$.

In the next proposition, we provide a more intuitive definition of $\mcvar_p(\vX,\veta)$ for the case that $\p(\vX \leq \veta) = p$.
\begin{proposition} \label{prop:altcvar} For a random variable $\vX \in \R^d$ and $\veta \in \mvar_p(\vX) $ such that $\p(\vX \leq \veta) = p$, the following equality holds,
\begin{equation} \label{def:mcvar4}
\mcvar_p(\vX,\veta)=\veta+\frac{1}{1-p}\E[(\vX-\veta)_+]=\E[\max(\vX,\veta)|\vX \nleq \veta]
\end{equation}
where  $\vX \nleq \veta$ denotes a component-wise relation such that $X_i > \eta_i$ for at least one $i\in [d]$, and $\max(\vX,\veta)$ represents a vector whose $i^{th}$ component is equal to $\max(X_i,\eta_i)$. 
\end{proposition}
\begin{proof}
For a  criterion $i \in [d]$, the $i^{th}$ component of $\mcvar_p(\vX,\veta)$ is equal to,
\begin{align*}
(\mcvar_p(\vX,\veta))_i & = \eta_i + \frac{1}{1-p} \E[(X_i-\eta_i)_+] \\
& = \eta_i + \frac{1}{\p(\vX \nleq \veta)} \E[(X_i-\eta_i)_+]\\
& = \eta_i +  \E[(X_i-\eta_i)_+|\vX \nleq \veta]\\
& = \E[\eta_i|\vX \nleq \veta] + \E[(\max(X_i,\eta_i)-\eta_i)|\vX \nleq \veta]\\
& = \E[(\eta_i+ \max(X_i,\eta_i)-\eta_i)|\vX \nleq \veta]\\
& = \E[\max(X_i,\eta_i)|\vX \nleq \veta]. 
\end{align*}
\end{proof}
This is also analogous to the univariate definition of CVaR since for a univariate random variable $V$ such that $\p(V\leq \var_p(V))=p$,
\begin{equation}\label{eq:unicvar}
\begin{split}
\cvar_p(V) &= \E(V|V > \var_p(V)) = \E(\max(V,\var_p(V))|V > \var_p(V))\\
& =  \E(\max(V,\var_p(V))|V \nleq \var_p(V)).
\end{split}
\end{equation}

The proposed multivariate risk measure should satisfy some desired properties to be a useful tool for evaluating the risk of random vectors.  \citet{artzner1999coherent} axiomatize the desired properties of univariate risk measures under the \textit{coherence} concept. A coherent univariate risk measure is normalized, positively homogeneous, translation equivariant, monotone and subadditive. The risk measure of interest in this study, $\vmcvar_p(\vX)$, is a mapping from a $d$-dimensional random vector to a set of $d$-dimensional vectors. Hence, we aim to prove an analogous form of these properties for set-valued multivariate risk measures. Note that for two random vectors, $\vX$ and $\vY$, of the same finite probability space, $\vX \leq \vY$ means that $x_i^s \leq y_i^s$ for each scenario $s\in [n]$ and criterion $i \in [d]$.

\begin{proposition}\label{prop:cohere}
For a $d$-dimensional random vector $\vX$, $\vmcvar_p(\vX)$ satisfies the following properties,
\begin{itemize}
\item[(i)] Normalized: $\vmcvar_p(\mathbf 0)= \{\mathbf 0\}.$
\item[(ii)] Positively homogeneous: $\forall \vvartheta \in \vmcvar_p(\vX), k\vvartheta \in  \vmcvar_p(k\vX)$ for $k \in \R_+$.
\item[(iii)] Translation equivariant: $\forall \vvartheta \in \vmcvar_p(\vX), \vvartheta + \vk \in \vmcvar_p(\vX + \vk) $ for $\vk \in \R^d$.
\item[(iv)] Monotone: $\vX \leq \vY \Rightarrow \forall \vvartheta_1 \in \vmcvar_p(\vX), \exists \vvartheta_2 \in \vmcvar_p(\vY)\text{ s.t. }\vvartheta_2 \geq \vvartheta_1. $
\end{itemize}
\end{proposition}

\begin{proof}
It can be easily seen that $\vmcvar_p(\vX)$ is normalized. We will prove the remaining properties using the fact that for all $\vvartheta \in \vmcvar_p(\vX)$, $\vvartheta = \mcvar_p(\vX,\veta)$ for some $\veta \in \mvar_p(\vX)$.
\begin{itemize}
\item[(ii)] From \cite{lee2013properties}, we know if $\veta \in \mvar_p(\vX)$, then $k\veta \in \mvar_p(k\vX)$ for $k \in \R_+$. Hence,
\begin{align*}
k \mcvar_p(\vX, \veta)  &=k\left(\veta+\frac{1}{1-p}\E[(\vX-\veta)_+]\right)     \\
&=k\veta+\frac{1}{1-p}\E[(k\vX-k\veta)_+] =\mcvar_p(k\vX,k\veta).
\end{align*}
\item[(iii)] \citet{lee2013properties} show that $\veta + \vk \in \mvar_p(\vX + \vk)$  for $\vk \in \R^d$ when $\veta \in \mvar_p(\vX)$. Using this,
\begin{align*}
\mcvar_p(\vX, \veta)  +\vk &=\veta+\frac{1}{1-p}\E[(\vX-\veta)_+] +\vk    \\
&=\veta+\vk+\frac{1}{1-p}\E[(\vX+\vk-\veta-\vk)_+] = \mcvar_p(\vX+\vk,\veta+\vk) .
\end{align*}
\item[(iv)] By the monotonicity property of MVaR given in \cite{lee2013properties}, $\vX \leq \vY$ implies that for all $\veta_1 \in \mvar_p(\vX)$ there exists $\veta_2 \in \mvar_p(\vY)$ such that $\veta_2 \geq \veta_1$ and there is no $\veta_2 \in \mvar_p(\vY)$ such that $\veta_2 < \veta_1$ for some $\veta_1 \in \mvar_p(\vX)$. Consider the set of undesirable scenarios for $\vX$, denoted by $S:=\{s\in[n]:\vx^s  \nleq \veta_1\}$ partitioned into two subsets: $S_1=\{s\in[n]:\vx^s  \nleq \veta_1, \vy^s<\veta_2\}$, and $S_2=\{s\in[n]:\vx^s  \nleq \veta_1, \vy^s\nless\veta_2\}$. Using this definition, 
\begin{align}
\vvartheta_1 &= \mcvar_p(\vX, \veta_1)=\veta_1+\frac{1}{1-p}\sum_{s\in[n]}\E(\vX-\veta_1)_+=\veta_1+\frac{1}{1-p}\sum_{s\in S}q_s(\vx^s-\veta_1)_+ \notag \\
&=\veta_1+\frac{1}{1-p}\sum_{s\in S_1} q_s(\vx^s-\veta_1)_+ + \frac{1}{1-p}\sum_{s\in S_2} q_s(\vx^s-\veta_1)_+ \label{eq:mon1}\\
& \le \veta_1+ \frac{1}{1-p}\sum_{s\in S_1} q_s(\veta_2-\veta_1) + \frac{1}{1-p}\sum_{s\in S_2} q_s(\vx^s-\veta_2+\veta_2-\veta_1)_+  \label{eq:mon2}\\
& \le \veta_1+ \frac{1}{1-p}\sum_{s\in S} q_s(\veta_2-\veta_1) + \frac{1}{1-p}\sum_{s\in S_2} q_s(\vx^s-\veta_2)_+  \label{eq:mon3}\\
&\le \veta_2+\frac{1}{1-p}\sum_{s\in [n]} q_s(\vy^s-\veta_2)_+= \mcvar_p(\vY, \veta_2)= \vvartheta_2, \label{eq:mon4}
\end{align}
where equality \eqref{eq:mon1} follows from the definition of sets $S_1,S_2$; inequality \eqref{eq:mon2} follows from the fact that for $s\in S_1$, $\vx^s\le \vy^s< \veta_2$;  inequality \eqref{eq:mon3} follows from the facts that $(\va+\vb)_+\le (\va)_+ + (\vb)_+$ for $\vb\ge \mathbf 0$ (here $\vb$ is taken as $\veta_2-\veta_1\ge \mathbf 0$); and inequality \eqref{eq:mon3} follows from the fact that $\sum_{s\in S}q_s\le 1-p$ and $\vx^s\le \vy^s$. 

\ignore{
, since otherwise there exists a $p$LEP strictly more favorable than $\veta_2$. 
Let $\vvartheta_1 = \mcvar_p(\vX, \veta_1) $ and $\vvartheta_2 = \mcvar_p(\vY, \veta_2) $. For scenario $\omega$, assume that $\vX(\omega)$ is an undesirable outcome with respect to $\veta_1$, i.e. $\vX(\omega) \nless \veta_1$ so that it is included in the computation of $\vvartheta_1$. There could be two cases:
\begin{itemize}
\item[-] $\vY(\omega)<\veta_2$, meaning that $\vX(\omega)<\veta_2$ since $\vX \leq \vY$. As $\vvartheta_2 \geq \veta_2$ by definition, $\vX(\omega) $ can not contribute enough to make $\vvartheta_1$ larger than or equal to $\vvartheta_2$.
\item[-] $\vY(\omega) \nless \veta_2$, implying realization $\vY(\omega)$ will be included in the computation of $\vvartheta_2$. Since $\vX \leq \vY$, the contribution of $\vY$ to $\vvartheta_2$ will be at least as much as the contribution of $\vX$ to $\vvartheta_1$.

\end{itemize}
Considering these cases, we  conclude that $\vvartheta_2$ is dominated by $\vvartheta_1$.}
\end{itemize}
\end{proof}

Note that \citet{lee2013properties}  show that the properties in Proposition \ref{prop:cohere} hold for their definition of MCVaR (we will give this definition in detail in Section \ref{sec:comp}). While we use the same definition of MVaR as \cite{lee2013properties}, our definition of MCVaR is different, hence the need to prove these properties for our definition. 
 \citet{lee2013properties} also show that a multivariate analogue of the subadditivity property, which is included in the coherence definition of univariate risk measures,  is not satisfied by their MCVaR definition.  We consider this property next. 

\begin{remark}
Risk measure $\mcvar$ does not satisfy the following analogue of the subadditivity property:
\begin{align*}
& \mcvar_p(\vX, \veta_x) + \mcvar_p(\vY, \veta_y) \geq \mcvar_p(\vX+\vY, \veta_{x+y}),\nonumber \\
&\hspace{3cm} \forall \veta_x \in \vmcvar_p(\vX), \veta_y \in \vmcvar_p(\vY), \veta_{x+y} \in \vmcvar_p(\vX+\vY).
\end{align*}
\end{remark}
\begin{proof}
Consider the following counterexample. Let $\vX$ and $\vY$ be bivariate random vectors with supports 
$$\vX\in \{(1,5),(3,2),(2,1),(1,4),(5,5) \} \text{ and } \vY \in\{(4,1.5),(1,3),(2,5),(2,3),(3,1)\},$$
such that each realization is equally likely. At confidence level $p=0.6$, $\mvar_p(\vX)=\{(3,4),(2,5)\}$, $\mvar_p(\vY)=\{(3,3),(2,5)\}$ and $\mvar_p(\vX+\vY)=\{(4,7),(8,6)\}$. Using this information, we compute $\vmcvar_p(\vX)=\{(4,5)\}$, $\vmcvar_p(\vY)=\{(3.5,4)\}$ and $\vmcvar_p(\vX+\vY)=\{(6.5,7),(8,6.75)\}$. It can be seen that $(4,5)+ (3.5,4) \ngeq (8,6.75)$, hence subadditivity property is violated.
\end{proof}

In addition to coherence, from  definition \eqref{cvardef}, it is clear that in the univariate case, CVaR is a conservative approximation of VaR, as $\cvar_p(V)\ge \var_p(V)$ for a univariate random variable $V$. 
 Analogously, it is expected that the multivariate CVaR value should not dominate its corresponding $p$LEP, i.e., we should not have $\mcvar_p(\vX,\veta) < \veta$ for any $\veta \in \mvar_p(\vX)$ for a random vector $\vX \in \R^d$. By definition (\ref{def:mcvar1}), $\vmcvar_p(\vX)$ clearly satisfies this property. 
 %since $\mcvar_p(\vX,\veta)$ is computed by taking expectation of $\max(X_i,\eta_i)$ for all $i\in [d]$, where $\veta \in \mvar_p(\vX)$. 
 In the next section, we show that this may not be the case for some existing definitions of multivariate CVaR.
 
\section{Comparison to Existing Definitions} \label{sec:comp}

In this section, we briefly review other multivariate CVaR definitions from the literature and compare them with the proposed definition of MCVaR.

\citet{prekopa2012multivariate} utilizes the definition of MVaR given in Definition \ref{def:plep} as a basis for their multivariate CVaR definition.
Using the definition of $\mvar_p(\vX)$, they assume an outcome $\vX$ to be desirable if 
\begin{equation} \label{favor}
\vX \in \bigcup\limits_{\vs \in \mvar_p(\vX)}(\vs + \R_-^d)
\end{equation}
and undesirable if
\begin{equation*}
\vX \in \bigcap\limits_{\vs \in \mvar_p(\vX)}(\vs + \R_-^d)^c,
\end{equation*}
where $A^c$ is the complement of set $A$. These definitions imply that an event is undesirable if it is not better than or equivalent to any element in $\mvar_p(\vX)$. Here, we assume that set $\R_-^d$ includes the zero vector along the same lines with the univariate definition of $\var$. Assume that the set of desirable outcomes, defined in the right-hand-side of (\ref{favor}), is denoted as $D_p$. The set $\overline{D_p^c}$ represents the closure of its complement, which is the set of undesirable outcomes.

\begin{definition}[\citet{prekopa2012multivariate}] \label{def:mcvarPrek}
The Multivariate Conditional Value-at-Risk of the random vector $\vX$ at confidence level $p$, denoted by  $\overline{\mcvar}_p(\vX)$, is   defined as
\begin{equation*}
\overline{\mcvar}_p(\vX) = \E[\psi(\vX)|\vX \in \overline{D_p^c}],
\end{equation*}
where $\psi$ is a $d$-variate function such that $\E[\psi(\vX)]$ exists.
\end{definition}

In particular, the function $\psi(\vz)$ in \cite{prekopa2012multivariate} is defined as 
$\psi(\vz) = \sum\limits_{i=1}^d c_i z_i, $
where $\sum_{i=1}^d c_i= 1$ and $\vc\geq \mathbf 0$. This scalarization function is consistent with the weighted sum approach commonly applied in multicriteria optimization, where $c_i$ is interpreted as the relative weight of criterion $i\in [d]$. Using the expansion on conditional probability, we know that
\begin{equation*}
\E[\psi(\vX)] = \E[\psi(\vX)|\vX \notin D_p]\p(\vX \notin D_p) + \E[\psi(\vX)|\vX \in D_p]\p(\vX \in D_p)
\end{equation*}
and consequently we obtain,
\begin{equation*}  %\label{formula:cvar_appr}
\E[\psi(\vX)|\vX \notin D_p]  = \frac{1}{\p(\vX \notin D_p)} (\E[\psi(\vX)]- \E[\psi(\vX)|\vX \in D_p]\p(\vX \in D_p)).
\end{equation*}
Assuming that the probability of $\p(\vX \in \mvar_p(\vX))$ is negligible, this equation can be used to estimate the value of $\overline{\mcvar}_p(\vX)$. This assumption is quite natural in case of continuous distributions. For discrete probability distributions, it  provides an approximation.

A shortcoming of  this definition is that $\overline{\mcvar}_p(\vX)$ fails to determine the risk associated with an outcome vector at any level $p$ in some discrete cases, which we describe next. 

\begin{example}\label{ex:last2}
Consider bivariate random variable $\vX \in \R^2$ with support $$\{(1,5),(2,4),(3,3),(4,2),(5,1) \}.$$ All realizations are assumed to have equal probability, 0.2. For different values of $p$, the set of desirable outcomes is depicted in Figure \ref{fig:Ex:all} as the gray shaded areas. The black and red points represent the realizations and $p$LEPs of $\vX$, respectively. As it can be seen, all possible realizations of $\vX$ are desirable for each possible value $0.4 \le  p \le 0.8$. Consequently, $\overline{\mcvar}_p(\vX)$ is undefined at any level $0.4 \le  p \le 0.8$.

\begin{figure}[h]
\includegraphics[width=8cm]{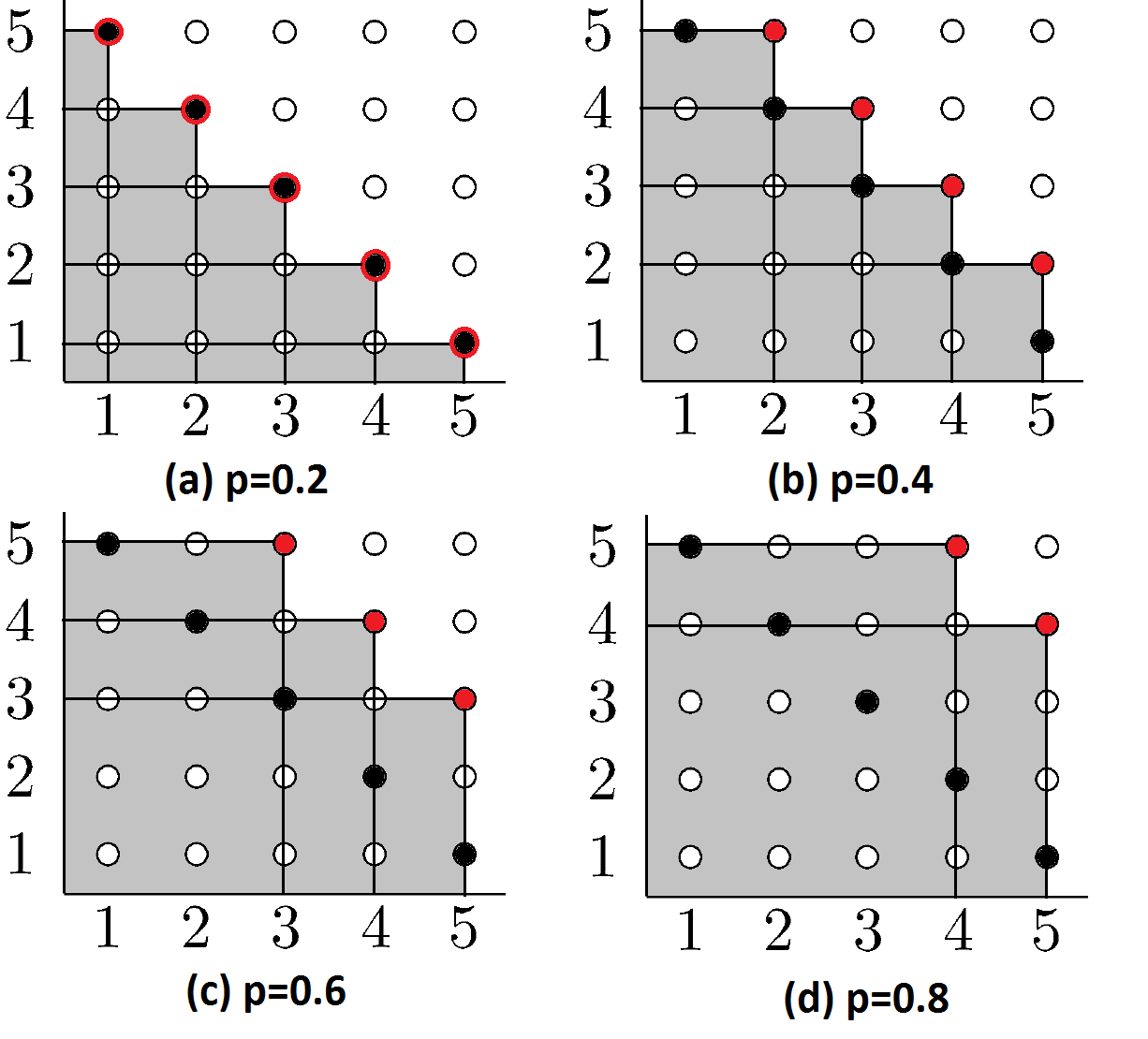}
\centering
\caption{Desirable outcomes for different values of $p$ in Example \ref{ex:last2}}
\label{fig:Ex:all}
\end{figure} 

\end{example}

Another problem is the possibility of obtaining an $\overline{\mcvar}_p(\vX)$ value that is preferable  to $\psi(\veta)$ for some $\veta \in \mvar_p(\vX)$. Consider the following example.

\begin{example}
Let $\vX \in \R^2$ be a bivariate random variable with support 
$$ \{(1,5),(2,4),(3,3),(4,2),(5,1) \},$$
such that $\p(\vX=(1,5))= \p(\vX=(5,1))=0.05$ and all the other realizations are equally likely.  At confidence level $p=0.9$, we obtain $\mvar_p(\vX) = \{(4,4)\}$. For $p$LEP (4,4), the realizations (1,5) and (5,1) are undesirable. The $\overline{\mcvar}$ value corresponding to this $p$LEP is $\psi((3,3))$, which is smaller than $\psi((4,4))$,  for any increasing function $\psi(\cdot)$ including the linear function proposed in  \cite{prekopa2012multivariate}.
\end{example} 

To the best of our knowledge, the only vector-valued alternative to the $\overline{\mcvar}$ of \citet{prekopa2012multivariate} for defining multivariate CVaR is proposed by \citet{cousin2014multivariate} under the name lower-orthant conditional tail expectation (\underline{CTE}). 
\begin{definition}[\citet{cousin2014multivariate}]
For a given random vector $\vX \in \R^d$ with distribution function $F$ at level $p$, the lower-orthant conditional tail expectation is defined as 
\begin{equation}\label{def:lowcte}
\text{\underline{CTE}}_p (\vX) = \E[\vX|F(\vX) \geq p].
\end{equation} 
\end{definition}

This definition includes only the outcomes dominated by or equal to a $p$LEP in the computation of multivariate CVaR, whereas the $\overline{\mcvar}$ of \citet{prekopa2012multivariate} considers the outcomes worse than all $p$LEPs in at least one criterion. A possible issue with this definition is that it may potentially lead to a conservative set of undesirable outcomes and consequently cases in which the risk measure is undefined. For instance, the set of undesirable outcomes in Example \ref{ex:last2} based on this definition is empty for all confidence levels. Our proposed definition remedies this issue, because the values exceeding the $p$-quantile  in not all but some of the criteria are also considered as undesirable in the VMCVaR definition.

Next we compare  these risk measures with the proposed definition of MCVaR for the simple case of a random variable with a single $p$LEP. Assume that we are given a random vector $\vX \in \R^d$ such that $\veta$ is the only $p$LEP of $\vX$ at uncertainty level $p$, i.e. $ \mvar_p(\vX) = \{\veta\}$. By doing so, we ensure that all multivariate CVaR definitions correspond to a single vector and hence they can be compared.
Furthermore, we consider the vector-valued adaptation of  $\overline{\mcvar}_p(\vX)$, denoted by $\overline{\vmcvar}_p(\vX)$, where we replace the scalarization function $\psi(\vX)$ in   Definition \ref{def:mcvarPrek} with the vector $\vX$. In particular, we define
\begin{equation} \label{def:mcvar11}
\overline{\vmcvar}_p(\vX)=\E[\vX|X_i\ge \eta_i \text{ for some } i\in[d]].
\end{equation}
as the expected vector value of  undesirable outcomes as defined in \cite{prekopa2012multivariate}. As explained previously, this may result in a multivariate CVaR value that dominates $\veta$.  The following results would also hold for the case that all measures are scalarized using the same vector.

We  represent definition \eqref{def:lowcte}  equivalently as,
\begin{equation*}  %\label{def:mcvar22}
\text{\underline{CTE}}_p (\vX)=\E[\vX|X_i\geq \eta_i\ \forall i \in [d]].
\end{equation*}

\ignore{

A possible benchmark for multivariate CVaR would be considering the CVaR value of each criterion independently. In this benchmark, we evaluate the CVaR value of each criterion isolated from the effects of the others. In particular, we define
\begin{equation*}  %\label{def:mcvar33}
\underline{\mcvar}_p(\vX):=\big(\E[X_1|X_1\geq \var_p(X_1)],\ldots,\E[X_d|X_d\geq \var_p(X_d)]\big)^\top.
\end{equation*}
Intuitively, we would expect this vector to be less than or equal to the actual vectors of multivariate CVaR since it does not account for the value of different criteria in a realization simultaneously. 
}

Note that in the following discussion, when the set $\vmcvar_p(\vX)$ has a single element in it, we refer to that element as $\vmcvar_p(\vX)$, as well. 

\begin{proposition}\label{prop:comp1}
For a given random vector $\vX$ such that $ \mvar_p(\vX)=\{\veta\}$,  $\p(\vX \leq \veta) = p$, and $\vx^s\ne \veta$ for all $s\in [n]$, where $\vx^s$ are $n<\infty$ possible realizations of $\vX$ with probabilities $q_s$, the following relations hold,
$$\overline{\vmcvar}_p(\vX)\leq  \vmcvar_p(\vX) \leq \text{\underline{CTE}}_p(\vX),$$
when $\overline{\vmcvar}_p(\vX)$ and $\text{\underline{CTE}}_p(\vX)$ are well-defined for $\vX$.
\end{proposition}
\begin{proof}
Note that the conditions of the proposition ensure that the risk measures are comparable. 
We consider the case that all risk measures are well-defined for $\vX$ (recall that $ \text{\underline{CTE}}_p(\vX)$ and $\overline{\vmcvar}_p(\vX)$ may not be defined in some cases). Since $\overline{\vmcvar}_p(\vX)$ defined in \eqref{def:mcvar11} takes the expectation of  $\vX$ while $\vmcvar_p(\vX)$ takes the expectation of  $\max(\vX,\veta)$ (see Proposition \ref{prop:altcvar})   and the condition in $\vmcvar_p(\vX)$ given in Proposition \ref{prop:altcvar} is more restrictive, we have $\overline{\vmcvar}_p(\vX)\leq \vmcvar_p(\vX)$. To show that $\vmcvar_p(\vX)\leq \text{\underline{CTE}}_p(\vX)$, we represent $\text{\underline{CTE}}_p(\vX)$ equivalently as,
\begin{equation*}
\text{\underline{CTE}}_p(\vX)=\E[\max(\vX,\veta)|X_i\geq \eta_i\ \forall i \in [d]].
\end{equation*}
The first terms in both definitions are the same except for the conditions, $\text{\underline{CTE}}_p(\vX)$ conditions on the realizations with $\vX \geq \veta$, while $\vmcvar_p(\vX)$ conditions on  the realizations with  $\vX \nleq \veta$). Hence the claim follows. 

\ignore{

Next we need to determine the relationship between  $\underline{\mcvar}_p(\vX)$ and the other definitions. We know that $\eta_i \geq \var_p(X_i)$ for all $i \in [d]$. For each criterion $i\in [d]$, the realizations that are dominated by $\veta$, i.e. $\vX\leq\veta$, with $X_i \leq \var_p(X_i)$ are considered desirable in both $\underline{\mcvar}_p(\vX)$ and $\overline{\vmcvar}_p(\vX)$. 
In the computation of multivariate CVaR using the definition $\underline{\mcvar}_p(\vX)$,  the realizations
 dominated by $\veta$ with $\eta_i \geq X_i > \var_p(X_i)$ are included, whereas using the definition $\overline{\vmcvar}_p(\vX)$, 
  the realizations undominated by $\veta$ such that $X_i \leq \var_p(X_i)$ are included.
Hence, we have $(\overline{\vmcvar}_p(\vX))_i\leq (\underline{\mcvar}_p(\vX))_i$ for all $i \in [d]$. Similar arguments also hold for $\vmcvar_p(\vX)$, but this time, the realizations of the second type are replaced by $\max(\vX,\veta)$. Considering this, it is clear that $\underline{\mcvar}_p(\vX) \leq \vmcvar_p(\vX)$.
}
\end{proof}

In summary, we propose a new definition of a vector-valued multivariate CVaR that is consistent with its univariate counterpart. We compare this definition with alternative vector-valued definitions, and show its advantages.  However, we recognize the computational difficulties involved with obtaining this vector, especially when used in an optimization setting. In this case, we believe that the scalarization approaches in 
\cite{noyan2013optimization,liu2015robust,kuccukyavuz2016cut,noyan2016optimization,NMK17} provide a practical and reasonable estimation of the multivariate risk. 

\section*{Acknowledgment}
Simge K\"{u}\c{c}\"{u}kyavuz and Merve Merakl\i\ are supported, in part, by
National Science Foundation Grants 1732364 and 1733001. 

%\section*{References}
\bibliographystyle{abbrvnat}
\bibliography{mybibfile}

\begin{thebibliography}{18}
\providecommand{\natexlab}[1]{#1}
\providecommand{\url}[1]{\texttt{#1}}
\expandafter\ifx\csname urlstyle\endcsname\relax
  \providecommand{\doi}[1]{doi: #1}\else
  \providecommand{\doi}{doi: \begingroup \urlstyle{rm}\Url}\fi

\bibitem[Adrian and Brunnermeier(2016)]{adrian2016covar}
T.~Adrian and M.~K. Brunnermeier.
\newblock $\covar$.
\newblock \emph{The American Economic Review}, 106\penalty0 (7):\penalty0
  1705--1741, 2016.

\bibitem[Artzner et~al.(1999)Artzner, Delbaen, Eber, and
  Heath]{artzner1999coherent}
P.~Artzner, F.~Delbaen, J.-M. Eber, and D.~Heath.
\newblock Coherent measures of risk.
\newblock \emph{Mathematical finance}, 9\penalty0 (3):\penalty0 203--228, 1999.

\bibitem[Cousin and Di~Bernardino(2013)]{cousin2013multivariate}
A.~Cousin and E.~Di~Bernardino.
\newblock On multivariate extensions of value-at-risk.
\newblock \emph{Journal of Multivariate Analysis}, 119:\penalty0 32--46, 2013.

\bibitem[Cousin and Di~Bernardino(2014)]{cousin2014multivariate}
A.~Cousin and E.~Di~Bernardino.
\newblock On multivariate extensions of conditional-tail-expectation.
\newblock \emph{Insurance: Mathematics and Economics}, 55:\penalty0 272--282,
  2014.

\bibitem[Dentcheva et~al.(2000)Dentcheva, Pr{\'e}kopa, and
  Ruszczynski]{dentcheva2000concavity}
D.~Dentcheva, A.~Pr{\'e}kopa, and A.~Ruszczynski.
\newblock Concavity and efficient points of discrete distributions in
  probabilistic programming.
\newblock \emph{Mathematical Programming}, 89\penalty0 (1):\penalty0 55--77,
  2000.

\bibitem[Di~Bernardino et~al.(2015)Di~Bernardino, Fern{\'a}ndez-Ponce,
  Palacios-Rodr{\'\i}guez, and Rodr{\'\i}guez-Gri{\~n}olo]{di2015multivariate}
E.~Di~Bernardino, J.~Fern{\'a}ndez-Ponce, F.~Palacios-Rodr{\'\i}guez, and
  M.~Rodr{\'\i}guez-Gri{\~n}olo.
\newblock On multivariate extensions of the conditional value-at-risk measure.
\newblock \emph{Insurance: Mathematics and Economics}, 61:\penalty0 1--16,
  2015.

\bibitem[F{\'a}bi{\'a}n and Veszpr{\'e}mi(2007)]{fabian2007algorithms}
C.~I. F{\'a}bi{\'a}n and A.~Veszpr{\'e}mi.
\newblock Algorithms for handling $\cvar$-constraints in dynamic stochastic
  programming models with applications to finance.
\newblock 2007.

\bibitem[K{\"u}{\c{c}}{\"u}kyavuz and Noyan(2016)]{kuccukyavuz2016cut}
S.~K{\"u}{\c{c}}{\"u}kyavuz and N.~Noyan.
\newblock Cut generation for optimization problems with multivariate risk
  constraints.
\newblock \emph{Mathematical Programming}, 159\penalty0 (1-2):\penalty0
  165--199, 2016.

\bibitem[Lee and Pr{\'e}kopa(2013)]{lee2013properties}
J.~Lee and A.~Pr{\'e}kopa.
\newblock Properties and calculation of multivariate risk measures: $\mvar$ and
  $\mcvar$.
\newblock \emph{Annals of Operations Research}, 211\penalty0 (1):\penalty0
  225--254, 2013.

\bibitem[Liu et~al.(2017)Liu, K{\"u}{\c{c}}{\"u}kyavuz, and
  Noyan]{liu2015robust}
X.~Liu, S.~K{\"u}{\c{c}}{\"u}kyavuz, and N.~Noyan.
\newblock Robust multicriteria risk-averse stochastic programming models.
\newblock \emph{Annals of Operations Research}, pages 1--36, 2017.
\newblock \url{https://doi.org/10.1007/s10479-017-2526-z}.

\bibitem[Noyan and Rudolf(2013)]{noyan2013optimization}
N.~Noyan and G.~Rudolf.
\newblock Optimization with multivariate conditional value-at-risk constraints.
\newblock \emph{Operations Research}, 61\penalty0 (4):\penalty0 990--1013,
  2013.

\bibitem[Noyan and Rudolf(2016)]{noyan2016optimization}
N.~Noyan and G.~Rudolf.
\newblock Optimization with stochastic preferences based on a general class of
  scalarization functions.
\newblock Optimization Online,
  \url{http://www.optimization-online.org/DB_FILE/2016/09/5636.pdf}, 2016.

\bibitem[Noyan et~al.(2017)Noyan, Merakli, and K{\"u}{\c{c}}{\"u}kyavuz]{NMK17}
N.~Noyan, M.~Merakli, and S.~K{\"u}{\c{c}}{\"u}kyavuz.
\newblock Two-stage stochastic programming under multivariate risk constraints
  with an application to humanitarian relief network design.
\newblock Optimization Online
  \url{http://www.optimization-online.org/DB_FILE/2017/01/5804.pdf}, 2017.

\bibitem[Pr{\'e}kopa(1990)]{prekopa1990dual}
A.~Pr{\'e}kopa.
\newblock Dual method for the solution of a one-stage stochastic programming
  problem with random {RHS} obeying a discrete probability distribution.
\newblock \emph{Zeitschrift f{\"u}r Operations Research}, 34\penalty0
  (6):\penalty0 441--461, 1990.

\bibitem[Pr{\'e}kopa(2012)]{prekopa2012multivariate}
A.~Pr{\'e}kopa.
\newblock Multivariate value at risk and related topics.
\newblock \emph{Annals of Operations Research}, 193\penalty0 (1):\penalty0
  49--69, 2012.

\bibitem[Rockafellar and Uryasev(2000)]{rockafellar2000optimization}
R.~T. Rockafellar and S.~Uryasev.
\newblock Optimization of conditional value-at-risk.
\newblock \emph{Journal of risk}, 2:\penalty0 21--42, 2000.

\bibitem[Rockafellar and Uryasev(2002)]{rockafellar2002conditional}
R.~T. Rockafellar and S.~Uryasev.
\newblock Conditional value-at-risk for general loss distributions.
\newblock \emph{Journal of banking \& finance}, 26\penalty0 (7):\penalty0
  1443--1471, 2002.

\bibitem[Torres et~al.(2015)Torres, Lillo, and Laniado]{torres2015directional}
R.~Torres, R.~E. Lillo, and H.~Laniado.
\newblock A directional multivariate value at risk.
\newblock \emph{Insurance: Mathematics and Economics}, 65:\penalty0 111--123,
  2015.

\end{thebibliography}

\end{document}